\documentclass[12pt]{amsart}

\usepackage{fullpage}
\usepackage{setspace}
\usepackage[leqno]{amsmath}
\usepackage{amsfonts}
\usepackage{amssymb}
\usepackage{amsthm}
\usepackage{amssymb}
\usepackage{bbm}
\usepackage{color}
\theoremstyle{plain}
\newtheorem{theorem}{Theorem}[section]

\newtheorem{lemma}[theorem]{Lemma}

\theoremstyle{definition}

\newcommand{\R}{\mathbb{R}}
\newcommand{\C}{\mathbb{C}}
\newcommand{\ga}{\gamma}
\newcommand{\de}{\delta}

\newcommand{\be}{\beta}

\newcommand{\si}{\sigma}

\newcommand{\al}{\alpha}
\newcommand{\Aa}{A^\alpha}
\newcommand{\Ba}{B^\alpha}
\newcommand{\Xa}{X^\alpha}
\newcommand{\Ab}{A^\beta}
\newcommand{\Bb}{B^\beta}
\newcommand{\Xb}{X^\beta}
\newcommand{\Ac}{A^\gamma}
\newcommand{\Bc}{B^\gamma}
\newcommand{\Xc}{X^\gamma}
\DeclareMathOperator{\nr}{nilrad}
\DeclareMathOperator{\Ann}{Ann}

\title{Solvable Leibniz Algebras with Triangular Nilradical}

\author[Bosko-Dunbar]{Lindsey Bosko-Dunbar}
\address{Department of Mathematics, Spring Hill College\\
Mobile, AL 36608}
\email{lboskodunbar@shc.edu}

\author[Burke]{Matthew Burke}
\address{Department of Mathematics, Spring Hill College\\
Mobile, AL 36608}
\email{mjburke@email.shc.edu}

\author[Dunbar]{Jonathan D. Dunbar}
\address{Department of Mathematics, Spring Hill College\\
Mobile, AL 36608}
\email{jdunbar@shc.edu}

\author[Hird]{J.T. Hird}
\address{Department of Mathematics, West Virginia University, Institute of Technology\\
Montgomery, WV 25136}
\email{John.Hird@mail.wvu.edu}

\author[Stagg Rovira]{Kristen Stagg Rovira}
\address{Department of Mathematics, The University of Texas at Tyler\\
Tyler, TX 75799}
\email{kstagg@uttyler.edu}

\begin{document}


\doublespacing
\maketitle

\begin{abstract}
%%Here's what we're going to prove: list, list, list. Shebang!
%%We extend the classification of solvable Lie algebras with Heisenberg nilradical to classify solvable Leibniz algebras with Heisenberg nilradical.
%All solvable Lie algebras with Heisenberg nilradical have already been classified.  We extend this result to a classification of solvable Leibniz algebras with Heisenberg nilradical.  As an example, we show the complete classification of all real or complex Leibniz algebras whose nilradical is the 3-dimensional Heisenberg algebra.
A classification exists for Lie algebras whose nilradical is the triangular Lie algebra $T(n)$.  We extend this result to a classification of all solvable Leibniz algebras with nilradical $T(n)$.  As an example we show the complete classification of all Leibniz algebras whose nilradical is $T(4)$.
\end{abstract}

\section{Introduction}\label{intro}

% Multicitations: (\cite{ayupov}, \cite{barnesengel}, \cite{jacobsonleib}, \cite{patso})
% Alternative option: \cite{ayupov, barnesengel, jacobsonleib, patso}

Leibniz algebras were defined by Loday in 1993 \cite{loday, loday2}.  In recent years it has been a common theme to extend various results from Lie algebras to Leibniz algebras \cite{ao, ayupov, omirov}.  Several authors have proven results on nilpotency and related concepts which can be used to help extend properties of Lie algebras to Leibniz algebras.
%Of particular interest to us are those which study nilpotency and related concepts.
Specifically, variations of Engel's theorem for Leibniz algebras have been proven by different authors \cite{barnesengel, jacobsonleib} and Barnes has proven Levi's theorem for Leibniz algebras \cite{barneslevi}.  Additionally, Barnes has shown that left-multiplication by any minimal ideal of a Leibniz algebra is either zero or anticommutative \cite{barnesleib}.

In an effort to classify Lie algebras, many authors place various restrictions on the nilradical \cite{cs, nw, rw, wld}.  In \cite{tw}, Tremblay and Winternitz study solvable Lie algebras with triangular nilradical.  It is the goal of this paper to extend these results to the Leibniz setting.

Recent work has been done on classification of certain classes of Leibniz algebras \cite{aor, chelsie-allison, heisen, clok, clok2}.  In \cite{heisen}, a subset of the authors of this work found a complete classification of all Leibniz algebras whose nilradical is Heisenberg.  In particular, this includes a classification of all Leibniz algebras whose nilradical is the triangular Lie algebra $T(3)$, since $T(3)$ is the three-dimensional Heisenberg algebra.  For this reason our primary example will be Leibniz algebras whose nilradical is the triangular algebra $T(4)$.
%Especially useful for this paper is the work by Ca\~nete and Khudoyberdiyev \cite{ck} which classifies all non-nilpotent 4-dimensional Leibniz algebras over $\C$.  We recover their results for the 4-dimensional solvable Leibniz algebra over $\C$ with Heisenberg nilradical, and extend this case to classify such Leibniz algebras over $\R$.  Our result is not limited to only the 4-dimensional case, and can be used to classify all Leibniz algebras whose nilradical is Heisenberg.  In the last section we classify all (indecomposable) solvable Leibniz algebras whose nilradical is the 3-dimensional Heisenberg algebra, $H(1)$.

\section{Preliminaries}

A Leibniz algebra, $L$, is a vector space over a field (which we will take to be $\C$ or $\R$) with a bilinear operation (which we will call multiplication) defined by $[x,y]$ which satisfies the Leibniz identity
\begin{equation}\label{Jacobi}
[x,[y,z]] = [[x,y],z] + [y,[x,z]]
\end{equation}
for all $x,y,z \in L$.  In other words $L_x$, left-multiplication by $x$, is a derivation.  Some authors choose to impose this property on $R_x$, right-multiplication by $x$, instead.  Such an algebra is called a ``right'' Leibniz algebra, but we will consider only ``left'' Leibniz algebras (which satisfy \eqref{Jacobi}).  $L$ is a Lie algebra if additionally $[x,y]=-[y,x]$.

The derived series of a Leibniz (Lie) algebra $L$ is defined by $L^{(1)}=[L,L]$, $L^{(n+1)}=[L^{(n)},L^{(n)}]$ for $n\ge 1$.  $L$ is called solvable if $L^{(n)}=0$ for some $n$. The lower-central series of $L$ is defined by $L^2 = [L,L]$, $L^{n+1}=[L,L^n]$ for $n>1$. $L$ is called nilpotent if $L^n=0$ for some $n$.  It should be noted that if $L$ is nilpotent, then $L$ must be solvable.

The nilradical of $L$ is defined to be the (unique) maximal nilpotent ideal of $L$, denoted by $\nr(L)$.  It is a classical result that if $L$ is solvable, then $L^2 = [L,L] \subseteq \nr(L)$.  From \cite{mubar}, we have that
\begin{equation}\label{dimension}
\dim (\nr(L)) \geq \frac{1}{2} \dim (L).
\end{equation}

The triangular algebra $T(n)$ is the $\frac{1}{2}n(n-1)$-dimensional Lie algebra whose basis is the set of strictly upper-triangular matrices, $\{N_{ik} \vert 1 \leq i < k \leq n \}$ defined by multiplications
\begin{equation}\label{tri}
[N_{ik},N_{ab}]=\de_{ka}N_{ib} - \de_{bi}N_{ak}.
\end{equation}

The left-annihilator of a Leibniz algebra $L$ is the ideal $\Ann_\ell(L) = \left\{x\in L\mid [x,y]=0\ \forall y\in L\right\}$. Note that the elements $[x,x]$ and $[x,y] + [y,x]$ are in $\Ann_\ell(L)$, for all $x,y\in L$, because of \eqref{Jacobi}.

An element $x$ in a Leibniz algebra $L$ is nilpotent if both $(L_x)^n = (R_x)^n = 0$ for some $n$.  In other words, for all $y$ in $L$
\begin{equation*}
[x,\cdots[x,[x,y]]] = 0 = [[[y,x],x]\cdots,x].
\end{equation*}

A set of matrices $\{\Xa\}$ is called linearly nilindependent if no non-zero linear combination of them is nilpotent.  In other words, if
\begin{equation*}
X = \displaystyle\sum_{\al=1}^f c_\al \Xa,
\end{equation*}
then $X^n=0$ implies that $c_\al=0$ for all $\al$.  A set of elements of a Leibniz algebra $L$ is called linearly nilindependent if no non-zero linear combination of them is a nilpotent element of $L$.

\section{Classification}

Let $T(n)$ be the $\frac{1}{2}n(n-1)$-dimensional triangular (Lie) algebra over the field $F$ ($\C$ or $\R$) with basis $\{N_{ik} \vert 1 \leq i < k \leq n \}$ and products given by \eqref{tri}.  We will extend $T(n)$ to a solvable Leibniz algebra $L$ of dimension $\frac{1}{2}n(n-1) + f$ by appending linearly nilindependent elements $\{X^1, \ldots, X^f\}$.  In doing so, we will construct an indecomposable Leibniz algebra whose nilradical is $T(n)$.

We construct the vector $N = (N_{12} N_{23} \cdots N_{(n-1)n} N_{13} \cdots N_{(n-2)n} \cdots N_{1n})^T$ whose components are the basis elements of the nilradical ordered along consecutive off-diagonals ($N_{i(i+1)}$ in order, then $N_{i(i+2)}$ in order, \ldots).  Then since $[L,L] \subseteq \nr(L)$, the brackets of $L$ are given by \eqref{tri} and
\begin{align*}
%\label{Aalph}
[\Xa,N_{ik}] &= \Aa_{ik,pq} N_{pq}\\
%\label{Balph}
[N_{ik},\Xa] &= \Ba_{ik,pq} N_{pq}\\
%\label{XX}
[\Xa,\Xb] &= \si^{\al\be}_{pq} N_{pq}.
\end{align*}
using Einstein summation notation on repeated indices (from here onward), where $1\leq \al, \be \leq f$, $\Aa_{ik,pq}, \si^{\al \be}_{pq} \in F$.  Note that $\Aa \in F^{r \times r}$, $N \in T(n)^{r \times 1}$ where $r = \frac{1}{2}n(n-1)$.

To classify Leibniz algebras $L(n,f)$ we must classify the matrices $\Aa$ and $\Ba$ and the constants $\si^{\al \be}_{pq}$.  The Jacobi identities for the triples $\{\Xa, N_{ik}, N_{ab}\}$, $\{N_{ik}, N_{ab}, \Xa\}$, $\{N_{ik}, \Xa, N_{ab}\}$ with $1 \leq \al \leq f$, $1 \leq i < k \leq n$, $1 \leq a < b \leq n$ give us respectively
%\begin{equation}\tag{6a}\label{2.14}
%\de_{ka} \Aa_{ib,pq} N_{pq} - \de_{bi} \Aa_{ak,pq} N_{pq} + \Aa_{ik,bq} N_{aq} - \Aa_{ik,pa} N_{pb} - \Aa_{ab,kq} N_{iq} + \Aa_{ab,pi} N_{pk} = 0
%\end{equation}
%
%\begin{equation}\tag{6b}\label{un2.14}
%\de_{ka} \Ba_{ib,pq} N_{pq} - \de_{bi} \Ba_{ak,pq} N_{pq} + \Ba_{ik,bq} N_{aq} - \Ba_{ik,pa} N_{pb} - \Ba_{ab,kq} N_{iq} + \Ba_{ab,pi} N_{pk} = 0
%\end{equation}
%
%\begin{equation}\tag{6c}\label{2.14twist}
%\de_{ka} \Aa_{ib,pq} N_{pq} - \de_{bi} \Aa_{ak,pq} N_{pq} + \Aa_{ik,bq} N_{aq} - \Aa_{ik,pa} N_{pb} + \Ba_{ab,kq} N_{iq} - \Ba_{ab,pi} N_{pk} = 0
%\end{equation}
%\addtocounter{equation}{1}
\begin{align}
\tag{4a}\label{2.14}\de_{ka} \Aa_{ib,pq} N_{pq} - \de_{bi} \Aa_{ak,pq} N_{pq} + \Aa_{ik,bq} N_{aq} - \Aa_{ik,pa} N_{pb} - \Aa_{ab,kq} N_{iq} + \Aa_{ab,pi} N_{pk} &= 0\\
\tag{4b}\label{un2.14}\de_{ka} \Ba_{ib,pq} N_{pq} - \de_{bi} \Ba_{ak,pq} N_{pq} + \Ba_{ik,bq} N_{aq} - \Ba_{ik,pa} N_{pb} - \Ba_{ab,kq} N_{iq} + \Ba_{ab,pi} N_{pk} &= 0\\
\tag{4c}\label{2.14twist}\de_{ka} \Aa_{ib,pq} N_{pq} - \de_{bi} \Aa_{ak,pq} N_{pq} + \Aa_{ik,bq} N_{aq} - \Aa_{ik,pa} N_{pb} + \Ba_{ab,kq} N_{iq} - \Ba_{ab,pi} N_{pk} &= 0.
\end{align}
\addtocounter{equation}{1}

As a consequence of \eqref{2.14} and \eqref{2.14twist}, we also have that 
$$\Aa_{ab,pi} N_{pk} - \Aa_{ab,kq} N_{iq} = - (\Ba_{ab,pi} N_{pk} - \Ba_{ab,kq} N_{iq}).$$
Thus $\Aa_{ab,pi} = - \Ba_{ab,pi}$ if $p<i$ and $\Aa_{ab,kq} = - \Ba_{ab,kq}$ if $k<q$.  Therefore
\begin{equation}\label{2.14cor}
\Aa_{ab,ik} = - \Ba_{ab,ik} \quad\forall ab,ik \text{ except } ik=1n.
\end{equation}

Similarly the Jacobi identities for the triples $\{\Xa, \Xb, N_{ab}\}$, $\{\Xa, N_{ik}, \Xb\}$, $\{N_{ik}, \Xa, \Xb\}$ with $1 \leq \al, \be \leq f$ and $1 \leq i < k \leq n$ give us respectively
%\begin{equation}\tag{8a}\label{2.15}
%[\Aa, \Ab]_{ik,pq} N_{pq} = \si^{\al \be}_{kq} N_{iq} - \si^{\al \be}_{pi} N_{pk}
%\end{equation}
%
%\begin{equation}\tag{8b}\label{2.15twist}
%[\Aa, \Bb]_{ik,pq} N_{pq} = - (\si^{\al \be}_{kq} N_{iq} - \si^{\al \be}_{pi} N_{pk})
%\end{equation}
%
%\begin{equation}\tag{8c}\label{un2.15}
%(\Bb\Aa+\Ba\Bb)_{ik,pq} N_{pq} = \si^{\al \be}_{kq} N_{iq} - \si^{\al \be}_{pi} N_{pk}
%\end{equation}
%\addtocounter{equation}{1}
\begin{align}
\tag{6a}\label{2.15}[\Aa, \Ab]_{ik,pq} N_{pq} =& \phantom{-(}\si^{\al \be}_{kq} N_{iq} - \si^{\al \be}_{pi} N_{pk}\\
\tag{6b}\label{2.15twist}[\Aa, \Bb]_{ik,pq} N_{pq} =& - (\si^{\al \be}_{kq} N_{iq} - \si^{\al \be}_{pi} N_{pk})\\
\tag{6c}\label{un2.15}(\Bb\Aa+\Ba\Bb)_{ik,pq} N_{pq} =& \phantom{-(}\si^{\al \be}_{kq} N_{iq} - \si^{\al \be}_{pi} N_{pk}.
\end{align}
\addtocounter{equation}{1}
Unlike the Lie case, these give nontrivial relations for $f=1$ or $\al=\be$ when $f>1$.

The Jacobi identity for the triple $\{\Xa, \Xb, \Xc\}$ with $1 \leq \al, \be, \ga \leq f$ gives us
\begin{equation}\label{2.16}
\si^{\be \ga}_{pq} \Aa_{pq,ik} - \si^{\al \be}_{pq} \Bc_{pq,ik} - \si^{\al \ga}_{pq} \Ab_{pq,ik} = 0.
\end{equation}
Again, we do not require $\al, \be, \ga$ to be distinct, which in particular gives nontrivial relations for $f \geq 1$.

In order to simplify the matrices $\Aa$ and $\Ba$ and the constants $\si^{\al \be}_{pq}$ we will make use of several transformations which leave the commutation relations \eqref{tri} invariant.  Namely
\begin{itemize}
\item  Redefining the elements of the extension:
\begin{equation}\label{2.17}
\begin{split}
&\hspace{23pt}\Xa \longrightarrow \Xa + \mu^\al_{pq} N_{pq}, \quad \mu^\alpha_{pq} \in F \\
\Rightarrow
&\begin{cases}
\Aa_{ik,ab} \longrightarrow \Aa_{ik,ab} + \de_{kb}\mu^\al_{ai} - \de_{ia}\mu^\al_{kb}\\
\Ba_{ik,ab} \longrightarrow \Ba_{ik,ab} - \de_{kb}\mu^\al_{ai} + \de_{ia}\mu^\al_{kb}.
\end{cases}
\end{split}
\end{equation}

\item  Changing the basis of $\nr(L)$:
\begin{equation}\label{2.18}
\begin{split}
&\hspace{17pt}N \longrightarrow GN, \quad G \in GL(r,F) \\
\Rightarrow
&\begin{cases}
\Aa \longrightarrow G\Aa G^{-1}\\
\Ba \longrightarrow G\Ba G^{-1}.
\end{cases}
\end{split}
\end{equation}

\item  Taking a linear combination of the elements $\Xa$.
\end{itemize}
The matrix $G$ must satisfy certain restrictions discussed later in order to preserve the commutation relations \eqref{tri} of $\nr(L)$.

Note that $N_{1n}$ is not used in \eqref{2.17} since it commutes with all the elements in $\nr(L)$.  Since \eqref{2.16} gives relations between the matrices $\Aa$, $\Ba$ and the constants $\si^{\al \be}_{pq}$, the unused constant $\mu^\al_{1n}$ can be used to scale the constants $\si^{\al \be}_{pq}$ when $f \geq 2$:
\begin{equation}\label{2.19}
\begin{split}
\Xa &\longrightarrow \Xa + \mu^\al_{1n} N_{1n}, \quad \mu^\alpha_{1n} \in F \\
\Rightarrow \si^{\al \be}_{pq} &\longrightarrow \si^{\al \be}_{pq} + \mu^\be_{1n}\Aa_{1n,pq} + \mu^\al_{1n}\Bb_{1n,pq}.
\end{split}
\end{equation}
In this transformation $\Aa$ is invariant, so we will be able to simplify some constants $\si^{\al \be}_{pq}$.

\section{Extensions of $T(4)$}

In this paper we will focus on triangular algebras $T(n)$ with $n\geq 4$ because:
\begin{itemize}
\item $T(2)$ is a one-dimensional algebra (hence by \eqref{dimension} $L$ has dimension at most 2) and the Jacobi identity gives that the only family of solvable non-Lie Leibniz algebras with one-dimensional nilradical are given by $L(c)= \langle a,b \rangle$ with $[a,a]=[a,b]=0$, $[b,a]=ca$, $[b,b]=a$ where $c \neq 0 \in F$.
\item $T(3)$ is a Heisenberg Lie algebra, and Leibniz algebras with Heisenberg nilradical were classified in \cite{heisen}.
\end{itemize}

Now we will consider the case when $n = 4$.  In particular, $N = (N_{12} N_{23} N_{34} N_{13} N_{24} N_{14})^T$ and $r = 6$.

We can proceed by considering the relations in \eqref{2.14} and \eqref{un2.14} for $1 \leq i < k \leq 4$, $1 \leq a < b \leq 4$, $k \neq a$, $b \neq i$

\begin{align}
\tag{11a}\label{3.3}\Aa_{ik,bq} N_{aq} - \Aa_{ik,pa} N_{pb} - \Aa_{ab,kq} N_{iq} + \Aa_{ab,pi} N_{pk} &= 0\\
\tag{11b}\label{un3.3}\Ba_{ik,bq} N_{aq} - \Ba_{ik,pa} N_{pb} - \Ba_{ab,kq} N_{iq} + \Ba_{ab,pi} N_{pk} &= 0.
%\tag{14c}\label{3.3twist}\Aa_{ik,bq} N_{aq} - \Aa_{ik,pa} N_{pb} + \Ba_{ab,kq} N_{iq} - \Ba_{ab,pi} N_{pk} &= 0.
\end{align}
\addtocounter{equation}{1}
Similarly for $1 \leq i < k = a < b \leq 4$, we obtain

\begin{align}
\tag{12a}\label{3.2}\Aa_{ib,pq} N_{pq} + \Aa_{ik,bq} N_{kq} - \Aa_{ik,pk} N_{pb} - \Aa_{kb,kq} N_{iq} + \Aa_{kb,pi} N_{pk} &= 0\\
\tag{12b}\label{un3.2}\Ba_{ib,pq} N_{pq} + \Ba_{ik,bq} N_{kq} - \Ba_{ik,pk} N_{pb} - \Ba_{kb,kq} N_{iq} + \Ba_{kb,pi} N_{pk} &= 0.
%\tag{15c}\label{3.2twist}\Aa_{ib,pq} N_{pq} + \Aa_{ik,bq} N_{kq} - \Aa_{ik,pk} N_{pb} + \Ba_{kb,kq} N_{iq} - \Ba_{kb,pi} N_{pk} &= 0.
\end{align}
\addtocounter{equation}{1}

Using the linear independence of the $N_{ik}$ with equation \eqref{3.3}, we can obtain relationships among the entries of the matrices $\Aa$, summarized in the matrix below.  For example, letting $ik = 12$ and $ab = 34$, the coefficient of $N_{14}$ gives $\Aa_{12,13} + \Aa_{34,24} = 0$, the coefficient of $N_{13}$ gives $\Aa_{34,23} = 0$, and the coefficient of $N_{24}$ gives $\Aa_{12,23} = 0$.
Using equation \eqref{un3.3}, we obtain the same relationships among the entries of $\Ba$.

\begin{align*}
\Aa =& 
\begin{pmatrix}
*&0&\Aa_{12,34}&\Aa_{12,13}&*&*\\
0&*&0&*&*&*\\
\Aa_{34,12}&0&*&*&-(\Aa_{12,13})&*\\
0&0&0&*&(\Aa_{12,34})&*\\
0&0&0&(\Aa_{34,12})&*&*\\
0&0&0&0&0&*
\end{pmatrix}
\end{align*} 

Applying \eqref{3.2} and \eqref{un3.2} in the same way, $\Aa$ becomes:

\begin{align*}
\Aa =& 
\begin{pmatrix}
\Aa_{12,12}&0&0&\Aa_{12,13}&*&*\\
&\Aa_{23,23}&0&\Aa_{23,13}&\Aa_{23,24}&*\\
&&\Aa_{34,34}&*&-(\Aa_{12,13})&*\\
&&&\Aa_{12,12} + \Aa_{23,23}&0&(\Aa_{23,24})\\
&&&&\Aa_{23,23} + \Aa_{34,34}&(\Aa_{23,13})\\
&&&&&\Aa_{12,12} + \Aa_{23, 23} + \Aa_{34,34}
\end{pmatrix}
\end{align*} 
As before, $\Ba$ has the same form above.  This with \eqref{2.14cor} implies that $\Ba_{13,14} = \Ba_{23,24} = -\Aa_{23,24} = -\Aa_{13,14}$.  Similarly, $\Ba_{24,14} = -\Aa_{24,14}$ and $\Ba_{14,14} = -\Aa_{14,14}$.
 
We can further simplify matrix $\Aa$ by performing the transformation specified in \eqref{2.17}.  Choosing $\mu^\al_{12} = -\Aa_{23,13}$, $\mu^\al_{23} = \Aa_{12,13}$, $\mu^\al_{34} = \Aa_{23,24}$, $\mu^\al_{13} = -\Aa_{34,14}$, and $\mu^\al_{24} = -\Aa_{12,14}$ leads to
$$\Aa_{23,13} = \Aa_{12,13} = \Aa_{23,24} = \Aa_{34,14} = \Aa_{12,14} = 0.$$ 

This gives us the matrices

\begin{align*}
\Aa =& 
\begin{pmatrix}
\Aa_{12,12}&0&0&0&\Aa_{12,24}&0\\
&\Aa_{23,23}&0&0&0&\Aa_{23,14}\\
&&\Aa_{34,34}&\Aa_{34,13}&0&0\\
&&&\Aa_{13,13}&0&0\\
&&&&\Aa_{24,24}&0\\
\phantom{-\Aa_{12,12}}&\phantom{-\Aa_{23,23}}&\phantom{-\Aa_{34,34}}&\phantom{-\Aa_{13,13}}&\phantom{-\Aa_{24,24}}&\,\,\Aa_{14,14}\,\,
\end{pmatrix}\\
\Ba =&
\begin{pmatrix}
-\Aa_{12,12}&0&0&0&-\Aa_{12,24}&\Ba_{12,14}\\
&-\Aa_{23,23}&0&0&0&\Ba_{23,14}\\
&&-\Aa_{34,34}&-\Aa_{34,13}&0&\Ba_{34,14}\\
&&&-\Aa_{13,13}&0&0\\
&&&&-\Aa_{24,24}&0\\
&&&&&-\Aa_{14,14}
\end{pmatrix}
\end{align*}

$$\Aa_{ik,ik} = \sum^{k-1}_{p = i}\Aa_{p(p+1),p(p+1)}.$$

Note that $\Aa_{12,12}$, $\Aa_{23,23}$, and $\Aa_{34,34}$ cannot simultaneously equal 0, otherwise the nilradical would no longer be $T(4)$.  The nilindependence among the $\Aa$ implies that $T(4)$ can have at most a three-dimensional extension, since there are three parameters on the diagonal.  

The form of the matrices $\Aa$ implies that the only nonzero elements of $[\Aa, \Ab]$ are
$$[\Aa, \Ab]_{12,24},\quad [\Aa, \Ab]_{23,14},\quad [\Aa, \Ab]_{34,13}.$$
The linear independence of the $N_{ik}$ with equation \eqref{2.15}, yields
\begin{align}
\label{Acommute}[\Aa,\Ab] =& \ 0\\
\label{3.12}[\Xa,\Xb] =& \ \si^{\al \be}_{14} N_{14}.
\end{align}
For example, letting $ik = 12$ in equation \eqref{2.15}, the coefficient of $N_{24}$ implies that $[\Aa,\Ab]_{12,24}=0$ and the coefficient of $N_{14}$ implies that $\si^{\al \be}_{24} = [\Aa,\Ab]_{12,14} = 0$.  Henceforth we will abbreviate $\si^{\al \be}_{14} = \si^{\al \be}$ as all other $\si^{\al \be}_{pq} = 0$.  Since the $\Aa$ commute by \eqref{Acommute}, \eqref{2.15} and \eqref{2.15twist} imply
\begin{equation}\label{ABcommute}
[\Aa,\Bb]=0.
\end{equation}  %A similar calculation shows that $[\Ba,\Bb]=0$.

Considering \eqref{ABcommute} componentwise, we find that $\Bb_{12,14}(\Aa_{14,14}-\Aa_{12,12}) = \Bb_{34,14}(\Aa_{14,14}-\Aa_{34,34}) = (\Ab_{23,14} + \Bb_{23,14})(\Aa_{14,14}-\Aa_{23,23}) = 0$.
Furthermore, by \eqref{ABcommute} and \eqref{un2.15}, $0=\Ba\Ab+\Bb\Ba=(\Ab+\Bb)\Ba$. Componentwise, this tells us that
\begin{equation}\label{LindseyLemma}
	0 = \Bb_{12,14}\Aa_{14,14} = \Bb_{34,14}\Aa_{14,14} = (\Ab_{23,14} + \Bb_{23,14})\Aa_{14,14}.
\end{equation}
In particular, if $\Bb$ has a nontrivial off-diagonal entry, then %$\Bb_{12,14}\ne0$, $\Bb_{34,14}\ne0$, or $\Ab_{23,14} \ne -\Bb_{23,14}$, for any $\be$, then, respectively,
\begin{equation}\label{offdiag}
\begin{cases}
\Bb_{12,14}\ne 0 &\Rightarrow \Aa_{12,12}=\Aa_{14,14}=0,\ \forall\alpha	\\
\Bb_{34,14}\ne 0 &\Rightarrow \Aa_{34,34}=\Aa_{14,14}=0,\ \forall\alpha	\\
\Bb_{23,14}\ne -\Ab_{23,14} &\Rightarrow \Aa_{23,23}=\Aa_{14,14}=0,\ \forall\alpha.
\end{cases}
\end{equation}

The form of the matrices $\Aa$ and $\Ba$ imply that \eqref{2.16} becomes
\begin{equation}\label{3.13}
\si^{\al \be}\Ac_{14,14}-\si^{\al \ga}\Ab_{14,14}+\si^{\be \ga}\Aa_{14,14}=0.
\end{equation}
By adding equations of form \eqref{3.13}, we get
\begin{equation}\label{Lieish}
(\si^{\be \ga} + \si^{\ga \be})\Aa_{14,14} = 0.
\end{equation}
For example $(\si^{12}+\si^{21})\Aa_{14,14}=0$ is obtained by adding \eqref{3.13} with $\be=1, \ga=2$ to \eqref{3.13} with $\be=2, \ga=1$.

On a related note, since $[\Xb,\Xb] \in \Ann_\ell(L)$, we have $0=[[\Xb,\Xb],\Xa]=[\si^{\be \be}N_{14},\Xa]= - \si^{\be \be} \Aa_{14,14} N_{14}$.  Thus,
\begin{equation}\label{Lieish2}
\si^{\be \be} \Aa_{14,14} = 0.
\end{equation}

As a consequence of \eqref{2.14cor}, \eqref{3.12}, \eqref{LindseyLemma}, \eqref{Lieish}, and \eqref{Lieish2}, we have the following result.
\begin{lemma}\label{megalem}
If $\Aa_{14,14}\ne0$ for any $\alpha=1,\ldots,f$, then the Leibniz algebra is a Lie algebra.
\end{lemma}

The form of the matrices $\Aa$ and $\Ba$ imply that the transformation \eqref{2.19} becomes
\begin{equation}\label{3.14}
\si^{\al \be} \longrightarrow \si^{\al \be} + \mu^\be_{14}\Aa_{14,14} - \mu^\al_{14}\Ab_{14,14}.
\end{equation}

For $f=2$, suppose $\Aa_{14,14} \neq 0$ for some $\al$, and without loss of generality assume that $\al=1$.  Then choosing $\mu^1_{14}=0$ and $\mu^2_{14}=-\frac{\si^{12}}{A^1_{14,14}}$; \eqref{3.14} makes $\si^{12}=0$.  For $f=3$, suppose $\Aa_{14,14} \neq 0$ for some $\al$, and without loss of generality assume that $\al=1$.  Then choosing $\mu^1_{14}=0$ and $\mu^\be_{14}=-\frac{\si^{1\be}}{A^1_{14,14}}$ for $\be=2,3$; \eqref{3.14} makes $\si^{1\be}=0$.  By \eqref{3.13}, we also have $\si^{23}=0$.  Combining these results for $f=2$, 3 and employing \eqref{Lieish2} for $f=1$, we have:
\begin{equation}\label{3.15}
[\Xa,\Xb] = 
\begin{cases}
\si^{\al \be}N_{14} & \text{if }A^1_{14,14}=\cdots=A^f_{14,14}=0 \\
0 & \text{otherwise.}
\end{cases}
\end{equation}

%Using \eqref{3.14} for $f=2$ and \eqref{3.14} and \eqref{3.13} for $f=3$, we find that
Now we will utilize matrices $G$ to simplify the structure of matrices $\Aa$ and $\Ba$.  Perform transformation \eqref{2.18}, given by $N\longrightarrow G_1 N$, with
$$G_1 = \begin{pmatrix}
1&0&0&0&g_1&g_0\\
 &1&0&0&0&g_2	\\
 & &1&g_3&0&g_4\\
 & & &1&0&0\\
 & & & &1&0\\
 & & & & &1
\end{pmatrix}.$$
Observe that $G_1$ acts invariantly on the commutation relations \eqref{tri}. It does, however, transform matrices $\Aa$ and $\Ba$ by $\Aa\longrightarrow G_1 \Aa G_1^{-1}$ and $\Ba\longrightarrow G_1 \Ba G_1^{-1}$, respectively.  In particular, $G_1$ transforms the following components
\begin{eqnarray*}
\begin{cases}
	\Aa_{12,24}	\longrightarrow	\Aa_{12,24} + g_1(\Aa_{24,24}-\Aa_{12,12})	\\
	\Aa_{23,14}	\longrightarrow	\Aa_{23,14} + g_2(\Aa_{14,14}-\Aa_{23,23})	\\
	\Aa_{34,13}	\longrightarrow	\Aa_{34,13} + g_3(\Aa_{13,13}-\Aa_{34,34})	
\end{cases}\\
\begin{cases}
	\Ba_{12,14}	\longrightarrow	\Ba_{12,14} - g_0(\Aa_{14,14}-\Aa_{12,12})	\\
	\Ba_{23,14}	\longrightarrow	\Ba_{23,14} - g_2(\Aa_{14,14}-\Aa_{23,23})	\\
	\Ba_{34,14}	\longrightarrow	\Ba_{34,14} - g_4(\Aa_{14,14}-\Aa_{34,34}).	
\end{cases}
\end{eqnarray*}
We use the matrix $G_1$ to eliminate some entries in $\Aa$ and $\Ba$. However, if $\Ba_{12,14}$ or $\Ba_{34,14}$ is not zero, then by \eqref{offdiag}, $G_1$ leaves that entry of $\Ba$ invariant. Hence, we can use $g_1$, $g_2$, and $g_3$ to eliminate at most 3 off-diagonal elements.

Note: In this step, we consider only transformations of the form $G_1=\begin{pmatrix} I&*\\0&I\end{pmatrix}$. Any other transformation which leaves \eqref{tri} and the form of $\Aa$ invariant, but eliminates $\Ba_{12,14}$ or $\Ba_{34,14}$, %could be used to transform
would provide an isomorphism from %
a non-Lie Leibniz algebra to a Lie algebra. %As non-Lie Leibniz algebras are not isomorphic to Lie algebras, we do not consider transformations which change $\Ba_{12,14}$ or $\Ba_{34,14}$ to 0.  because such a transformation would imply that there is a non-Lie Leibniz algebra which is isomorphic to a Lie algebra, which is impossible.  
It is, however, possible to scale such entries, which we will consider in the next case. %G_2 transformations

Let the diagonal matrix $G_2$ be 
$$G_2=\begin{pmatrix}
g_{12}&&&&&\\
&g_{23}&&&&\\
&&g_{34}&&&\\
&&&g_{12}g_{23}&&\\
&&&&g_{23}g_{34}&\\
&&&&&g_{12}g_{23}g_{34}
\end{pmatrix},\quad g_{ik}\in F\backslash\{0\}.$$
Note that $G_2$ preserves commutation relations \eqref{tri}.  Our transformation of $\nr(L)$ will be defined by $G=G_2 G_1$. Observe that $G_2$ transforms $\Aa$ and $\Ba$ by $\Aa_{ik,ab}\longrightarrow \dfrac{g_{ik}}{g_{ab}}\Aa_{ik,ab}$ and $\Ba_{ik,ab}\longrightarrow \dfrac{g_{ik}}{g_{ab}}\Ba_{ik,ab}$, respectively, where $g_{ik}=(G_2)_{ik}=\prod\limits^{k-1}_{j=i}g_{j(j+1)}$. Hence, we can scale up to three nonzero off-diagonal elements to 1. In the case of Lie algebras, it may be necessary to scale to $\pm 1$ over $F=\R$.  This issue does not arise in Leibniz algebras of non-Lie type, because we have greater restrictions on the number of nonzero entries.

\subsection{Leibniz algebras $L(4,1)$}
The Lie cases for $\nr(L)=T(4)$, with $f=1$, have been previously classified in \cite{tw}, so we will focus on the Leibniz algebras of non-Lie type. We know that all such algebras will have $A^1_{14,14}=0$ by Lemma \ref{megalem}, where $A=A^1$ will be of the form found in \cite{tw}. Since $A$ is not nilpotent and $A_{14,14}=0$, we know that there is at most 1 nonzero off-diagonal entry in $A$. Altogether, there are 10 classes of Leibniz algebras of non-Lie type. Of these, there are 2 two-dimensional families, and 8 one-dimensional families. The matrices $A$ and $B$ for these can be found in Table \ref{L41} in the appendix.

\subsection{Leibniz algebras $L(4,2)$}
Again, all Lie cases for $\nr(L)=T(4)$, with $f=2$, were classified in \cite{tw}.  So, focusing on Leibniz algebras of non-Lie type, we again require that $A^1_{14,14}=A^2_{14,14}=0$ by Lemma \ref{megalem}.  As a result, if $\Ba_{ik,14}\ne-\Aa_{ik,14}$, for any $\al$ or pair $ik$, then $\Ab_{ik,ik}=\Ab_{14,14}=0$ $\forall \be$, which makes it impossible to have two linearly nilindependent matrices $A^1$ and $A^2$. Therefore, there is 1 four-dimensional family of Leibniz algebras of non-Lie type, and their matrices $A^1=-B^1$ and $A^2=-B^2$ can be found in Table \ref{L42} in the appendix.

\subsection{Leibniz algebras $L(4,3)$}
There is only one Lie algebra that is a three-dimensional extension of $T(4)$, again given in \cite{tw}. Since we cannot have three linearly nilindependent matrices of form $\Aa$ with $A^1_{14,14} = A^2_{14,14} = A^3_{14,14} = 0$, it is impossible to have a three-dimensional extention of $T(4)$ that is of non-Lie type, by Lemma \ref{megalem}.

\begin{theorem}\label{L4f}
Every Leibniz algebra $L(4,f)$ is either of Lie type, or is isomorphic to precisely one algebra represented in Table \ref{L41} or Table \ref{L42}.
\end{theorem}

\section{Solvable Lie algebras $L(n,f)$ for $n\ge4$}

We are now going to consider Leibniz algebras $L$ with $\nr(L)=T(n)$. Recall from \eqref{2.14cor}, $\Aa_{ik,ab} = -\Ba_{ik,ab}$ for all $ab\ne 1n$. We have the following result:
\begin{lemma}\label{Astructure}
Matrices $\Aa=(\Aa_{ik,ab})$ and $\Ba=(\Ba_{ik,ab})$, $1\le i<k\le n$, $1\le a < b\le n$ have the following properties.
\begin{enumerate}
	\item[i.]	$\Aa$ and $\Ba$ are upper-triangular.
	\item[ii.]	The only off-diagonal elements of $\Aa$ and $\Ba$ which may not be eliminated by an appropriate transformation on $\Xa$ are:
	\begin{align*}\label{lem1.2A}
		\Aa_{12,2n},\quad \Aa_{j(j+1),1n} \ (2\le j\le n-2),\quad \Aa_{(n-1)n,1(n-1)},	\\
	%\end{align*}
	%\begin{align*}\label{lem1.2B}
		\Ba_{12,2n},\quad \Ba_{j(j+1),1n} \ (1\le j\le n-1),\quad \Ba_{(n-1)n,1(n-1)}.
	\end{align*}
	\item[iii.]	The diagonal elements $\Aa_{i(i+1),i(i+1)}$ and $\Ba_{i(i+1),i(i+1)}$, $1\le i\le n-1$, are free.  The remaining diagonal elements of $\Aa$ and $\Ba$ satisfy
		\begin{equation*}\label{lem1.3}
			\Aa_{ik,ik} = \sum\limits^{k-1}_{j=i}\Aa_{j(j+1),j(j+1)},\quad \Ba_{ik,ik} = \sum\limits^{k-1}_{j=i}\Ba_{j(j+1),j(j+1)},\quad k>i+1.
		\end{equation*}
\end{enumerate}
\end{lemma}
\begin{proof}
	The form of the matrices $\Aa$ given in Lemma \ref{Astructure} follows from \eqref{2.14} by induction on $n$, as shown in \cite{tw}. Similarly, properties $i.$ and $iii.$ follow for $\Ba$ from \eqref{un2.14}.  Property $ii.$ for matrices $\Ba$ follows from \eqref{2.14cor}. 
\end{proof}
As a consequence of property $iii.$ and \eqref{2.14cor}, we have that $\Aa_{1n,1n}=-\Ba_{1n,1n}$.
%%%%% Hey JT!  You said you would put in the matrix \Aa (and \Ba?) here!
%\begin{pmatrix}
%A_{12,12} &&& \cdots && A_{12,2n} & 0 \\
%& %A_{23,23}\cdots
% &&&&& A_{23,1n} \\
%& \ddots &&&&& \vdots \\
%& %A_{(n-2)(n-1),(n-2)(n-1)}
% &&&&& A_{(n-2)(n-1),1n} \\
%&& A_{(n-1)n,(n-1)n} & \cdots & A_{(n-1)n,1(n-1)} && 0 \\
%&&& \ddots &&& \vdots \\
%&&&& \ddots && \vdots \\
%&&&&& \ddots & 0 \\
%&&&&&& A_{1n,1n}
%\end{pmatrix}
%
%\begin{pmatrix}
%& A_{12,2n} & \\
%&& A_{23,1n} \\
%&& \vdots \\
%&& * \\
%&& \vdots \\
%&& A_{(n-2)n,1n} \\
%A_{(n-1)n,(n-1)n} &&
%\end{pmatrix}
%
%Lemma \ref{Astructure} asserts that the form of $\Ba$ is given below, and the form of $\Aa$ is the same with $b_1=b_2=0$. 
Lemma \ref{Astructure} asserts that $\Aa$ has $n-1$ free entries on the diagonal and nonzero off-diagonal entries possible in only $n-1$ locations, represented by $*$ in the matrix below.  The form of $\Ba$ is the same, save for nonzero off-diagonal entries possible in two additional locations, represented by $b_1$ and $b_2$ in the matrix below.

%\begingroup  %this seems to be changing all the matrices in the document
%\renewcommand*{\arraystretch}{.5}
%\setlength\arraycolsep{2pt}
%%\renewcommand\arraystretch{.5}
%%\setlength\arraycolsep{2pt}
%\begin{table}
$\hfill\left(
\begin{array}{ccccc|ccccc}
 * &&&&				&		&& &*&b_1 \\
& * &&&				&		&& &&* \\
&& \ddots &&	&		&& &&\vdots \\
&&& * &				&		&& && * \\
&&&& * 				&		&& *&&b_2 \\
\hline						
&&&&					&		*&&&&\\
&&&&					&		& \ddots &&& \\
&&&&					&		&&*&&\\
&&&&					&		&&&*&\phantom{\ddots}\\
\phantom{\ddots}&\phantom{\ddots}&\phantom{\ddots}&\phantom{\ddots}&\phantom{\ddots}&\phantom{\ddots}&\phantom{\ddots}&\phantom{\ddots}&\phantom{\ddots}&*
\end{array}
\right)
\hfill$
%\end{table}
%\endgroup

\begin{lemma}
The maximum degree of an extension of $T(n)$ is $f=n-1$.
\end{lemma}
\begin{proof}
	The proof follows from the fact that the $\Aa$ are nilindependent and that we have, at most, $n-1$ parameters along the diagonal.
\end{proof}

The form of the matrices $\Aa$ implies that the only nonzero elements of $[\Aa, \Ab]$ are
$$[\Aa,\Ab]_{12,2n},\quad [\Aa,\Ab]_{j(j+1),1n} \ (2\le j\le n-2),\quad [\Aa,\Ab]_{(n-1)n,1(n-1)}.$$
As before, the linear independence of the $N_{ik}$ with equations \eqref{2.15} and \eqref{2.15twist}, yields
\begin{align}
\label{Acommute2}[\Aa,\Ab] =&	\ 0\\
\label{ABcommute2}[\Aa,\Bb]=&	\	0\\
%\label{3.12b}
\nonumber [\Xa,\Xb] =& \ \si^{\al \be}_{1n} N_{1n}.
\end{align}

From Lemma \ref{Astructure}, \eqref{2.16} becomes
\begin{equation*}\label{3.13b}
\si^{\al \be}\Ac_{1n,1n}-\si^{\al \ga}\Ab_{1n,1n}+\si^{\be \ga}\Aa_{1n,1n}=0.
\end{equation*}

\begin{lemma}\label{commutelem}
Matrices $\Aa$ and $\Ba$ can be transformed to a canonical form satisfying
\begin{align*}%\label{}
[\Aa, \Ab] &= 0	\\
[\Aa, \Bb] &= 0	\\
[\Xa,\Xb] &= 
	\begin{cases}
		\si^{\al \be}N_{1n} & \text{if }A^1_{1n,1n}=\cdots=A^f_{1n,1n}=0 \\
		0 & \text{otherwise.}
	\end{cases}
\end{align*}
\end{lemma}
\begin{proof}
The first two identities of this lemma have already been shown.  The argument for the third identity is the same as \eqref{3.15}, using $(\si^{\be \ga} + \si^{\ga \be})\Aa_{1n,1n} = 0$ and $\si^{\be \be} \Aa_{1n,1n} = 0$.
\end{proof}

\subsection{Change of basis in $\nr(L(n,f))$}

As before, we perform the transformation \eqref{2.18} on $N$ by use of the matrix $G_1$, with $G_1$ defined to be all zeroes except $(G_1)_{ik,ik} = 1$, and $(G_1)_{12,1n}, (G_1)_{(n-1)n,1n}, (G_1)_{ab,ik} \in F$ where $\{ab,ik\}=\{12,2n\}, \{(n-1)n,1(n-1)\}, \{j(j+1),1n\}$ for $2 \leq j \leq n-2$.  Therefore, the zero entries of $G_1$ are the off-diagonal entries of $\Ba$ that are guaranteed to be zero. Note that the transformation given by $G_1$ preserves commutation relation \eqref{tri}.  Matrices $\Aa$ and $\Ba$ are transformed by conjugation with $G_1$, leaving the diagonal elements invariant and giving
\begin{eqnarray*}
%\begin{cases}
%	\Aa_{12,2n}	&\longrightarrow	\Aa_{12,2n} + g_1(\Aa_{2n,2n}-\Aa_{12,12})	\\
%	\Aa_{j(j+1),1n}	&\longrightarrow	\Aa_{j(j+1),1n} + g_j(\Aa_{1n,1n}-\Aa_{j(j+1),j(j+1)}),\quad j=2,\ldots, n-2	\\
%	\Aa_{(n-1)n,1(n-1)}	&\longrightarrow	\Aa_{(n-1)n,1(n-1)} + g_{n-1}(\Aa_{1(n-1),1(n-1)}-\Aa_{(n-1)n,(n-1)n})	
%\end{cases}\\
%\begin{cases}
%	\Ba_{12,2n}					&\longrightarrow	\Ba_{12,2n} - g_1(\Aa_{2n,2n}-\Aa_{12,12})	\\
%	\Ba_{j(j+1),1n}			&\longrightarrow	\Ba_{j(j+1),1n} - g_j(\Aa_{1n,1n}-\Aa_{j(j+1),j(j+1)}),\quad j=2,\ldots, n-2	\\
%	\Ba_{(n-1)n,1(n-1)}	&\longrightarrow	\Ba_{(n-1)n,1(n-1)} - g_{n-1}(\Aa_{1(n-1),1(n-1)}-\Aa_{(n-1)n,(n-1)n})	\\
%\end{cases}\\
%\begin{cases}
%	\Ba_{12,1n}	&\longrightarrow	\Ba_{12,1n} - g_0(\Aa_{1n,1n}-\Aa_{12,12})	\\
%	\Ba_{(n-1)n,1n}	&\longrightarrow	\Ba_{(n-1)n,1n} - g_n(\Aa_{1n,1n}-\Aa_{(n-1)n,(n-1)n}).	
%\end{cases}
\begin{split}
\Aa_{ik,ab}	&\longrightarrow	\Aa_{ik,ab} + (G_1)_{ik,ab}(\Aa_{ab,ab}-\Aa_{ik,ik})	\\
\Ba_{ik,ab}	&\longrightarrow	\Ba_{ik,ab} - (G_1)_{ik,ab}(\Aa_{ab,ab}-\Aa_{ik,ik}).
\end{split}
\end{eqnarray*}

By \eqref{ABcommute2}, we have that $0=[\Aa,\Bb]_{12,1n} = \Bb_{12,1n}(\Aa_{12,12}-\Aa_{1n,1n})$ and $0=[\Aa,\Bb]_{(n-1)n,1n} = \Bb_{(n-1)n,1n}(\Aa_{(n-1)n,(n-1)n}-\Aa_{1n,1n})$. Consequently, $G_1$ cannot eliminate the entries $\Bb_{12,1n}$ and $\Bb_{(n-1)n,1n}$.

\begin{lemma}\label{lemma4}
Matrices $\Aa$ and $\Ba$ will have a nonzero off-diagonal entry $\Aa_{ik,ab}$ or $\Ba_{ik,ab}$, respectively, only if 
$$\Ab_{ik,ik}=\Ab_{ab,ab},\quad \forall\be=1,\ldots,f.$$
\end{lemma}
\begin{proof}
For $\Aa$, the proof of Lemma \ref{lemma4} follows from \eqref{Acommute2}, as shown in \cite{tw}. Similarly for $\Ba$, the proof of Lemma \ref{lemma4} follows from considering \eqref{ABcommute2} componentwise. Namely, we find that $0=[\Ab,\Ba]_{12,1n}=\Ba_{12,1n}(\Ab_{1n,1n}-\Ab_{12,12})$, %and $0=[\Ab,\Ba]_{(n-1)n,1(n-1)}=\Ba_{(n-1)n,1(n-1)}(\Ab_{1(n-1),1(n-1)}-\Ab_{(n-1)n,(n-1)n})$.
and similar identities for $\Ba_{12,2n}$ and $\Ba_{(n-1)n,1(n-1)}$.
%Since Lemma \ref{lemma4} implies $\Aa_{23,1n}(\Ab_{1n,1n}-\Ab_{23,23})=0$, we have that for $\Ba_{23,1n}$: $$0=[\Ab,\Ba]_{23,1n} = (\Aa_{23,1n}+\Ba_{23,1n})(\Ab_{1n,1n}-\Ab_{23,23})=\Ba_{23,1n}(\Ab_{1n,1n}-\Ab_{23,23}).$$
The commutation relations for $[\Aa,\Bb]_{j(j+1),1n}$ gives that $\Ba_{j(j+1),1n}(\Ab_{1n,1n}-\Ab_{j(j+1),j(j+1)})=-\Ab_{j(j+1),1n}(\Aa_{1n,1n}-\Aa_{j(j+1),j(j+1)})=0$.
\end{proof}

Now consider a second transformation $G_2$ given by $N \longrightarrow G_2 N$, where $G_2$ is the diagonal matrix $(G_2)_{ik,ik} = g_{ik}$ and $g_{ik}=\prod\limits_{j=i}^{k-1} g_{j(j+1)}$.  The matrices $\Aa$ and $\Ba$ are transformed by conjugation by $G_2$.  Thus $\Aa_{ik,ab} \longrightarrow \dfrac{g_{ik}}{g_{ab}} \Aa_{ik,ab}$ and $\Ba_{ik,ab} \longrightarrow \dfrac{g_{ik}}{g_{ab}} \Ba_{ik,ab}$.  This transformation can be used to scale up to $n-1$ nonzero off-diagonal elements to 1.  For Lie algebras over the field $F=\R$ it may be that some entries have to be scaled to $-1$.

Since the only non-Lie cases occur when $\Ab_{1n,1n}=0$ for all $\be$, then $\Ba_{ik,1n} \ne -\Aa_{ik,1n}$ implies $\Ab_{ik,ik}=0$ for all $\be$ by Lemma \ref{lemma4}.  Since the extensions $\Xa$ are required to be nilindependent, this imposes restrictions on the degree $f$ of non-Lie extensions of $T(n)$.  In particular, this implies that in the maximal case $f=n-1$, $L(n,n-1)$ must be Lie.  Such algebras have been classified in \cite{tw}, and in fact there is a unique algebra $L(n,n-1)$ where all $\Aa$ are diagonal and the $\Xa$ commute.  %This imposes a whole bunch of beneficial restrictions via nilindependence (have at most n-2 nonzero off-diagonal elements, so n-1 parameters should have NO trouble scaling them down to 1).

\begin{theorem}
Every solvable Leibniz algebra $L(n,f)$ with triangular nilradical $T(n)$ has dimension $d=\frac{1}{2}n(n-1) + f$ with $1 \leq f \leq n-1$.  It can be written in a basis $\{\Xa, N_{ik}\}$ with $\al= 1, \ldots, f$, $1 \leq i < k \leq n$ satisfying
$$[N_{ik},N_{ab}]=\de_{ka}N_{ib} - \de_{bi}N_{ak}$$
$$[\Xa,N_{ik}]=\Aa_{ik,pq}N_{pq}$$
$$[N_{ik},\Xa]=\Ba_{ik,pq}N_{pq}$$
$$[\Xa,\Xb]=\si^{\al \be} N_{1n}.$$

Furthermore, the matrices $\Aa$ and $\Ba$ and the constants $\si^{\al \be}$ satisfy:
\begin{enumerate}
\item[i.] The matrices $\Aa$ are linearly nilindependent and $\Aa$ and $\Ba$ have the form specified in Lemma \ref{Astructure}.  $\Aa$ commutes with all these matrices, i.e. $[\Aa,\Ab]=[\Aa,\Bb]=0$.
\item[ii.] $\Ba_{ik,ab}=-\Aa_{ik,ab}$ for $ab \neq 1n$, and $\Ba_{1n,1n}=-\Aa_{1n,1n}$.
\item[iii.] $L$ is Lie and all $\si^{\al \be}=0$, unless $\Ac_{1n,1n}=0$ for $\ga=1, \ldots, f$. 
\item[iv.] The remaining off-diagonal elements $\Aa_{ik,ab}$ and $\Ba_{ik,ab}$ are zero, unless $\Ab_{ik,ik}=\Ab_{ab,ab}$ for $\be=1, \ldots, f$.
\item[v.] In the maximal case $f=n-1$, there is only one algebra, which is isomorphic to the Lie algebra with all $\Aa$ diagonal where all $\Xa$ commute.
\end{enumerate}
\end{theorem}

\vspace{12pt}

\noindent{\bf Acknowledgements.}  

%The authors gratefully acknowledge the support of the faculty summer research grant from Spring Hill College, as well as the support of the Departments of Mathematics at Spring Hill College, the University of Texas at Tyler, and West Virginia University: Institute of Technology.
The authors gratefully acknowledge the support of the Departments of Mathematics at Spring Hill College, the University of Texas at Tyler, and West Virginia University: Institute of Technology.

\newpage

\begin{table}[b]
\tiny
\caption{The Leibniz algebras $L(4,1)$ of non-Lie type}%bniz type}
\begin{tabular}{lllc}
\hline
%Name	&	
No.	&	$A$	&	$B$		&	$\si$, parameters	\\
\hline\hline
%\label{K11(a,-1-a)}$K_{1,1}(a,-1-a)$	&	
(1)
&	$%A = 
\left(\begin{array}{cccccc}
	\ 1	&		&	&&&	\\
		&	a	&	&&& 	\\
		&		&	-1-a &&&\\
		& 	& & 1+a && \\
		& 	& & &-1& \\
		& 	& & && 0\\
\end{array}\right)$	&	$B=-A$	&	
	$\begin{array}{l}\si^{11} \neq 0 \in F\\ a\ne-1 \in F\end{array}$	\\
\hline
%This is K_1,1(0,-1)
%\label{K11(0,-1)}$K_{1,1}(0,-1)$		&	
(2)	
&	$%A = 
\left(\begin{array}{cccccc}
	\ 1	&		&	&&&	\\
		&	0	&	\ &&&	\\
		&		&	-1\ 	&&&\\
		& 	& & 1 && \\
		& 	& &&-1& \\
		& 	& & && 0\\
\end{array}\right)$	&
$%B = 
\left(\begin{array}{cccccc}
	\ -1	&		&	&&&	\\
		&	0	&	\ &&&	1\\
		&		&	1\ 	&&&\\
		& 	& & -1 && \\
		& 	& & &1& \\
		& 	& & && 0\\
\end{array}\right)$	&
$\si^{11}\in F$		\\
\hline
%%%K_11(-1,0)
%\label{K11(-1,0)}$K_{1,1}(-1,0)$ &	
(3)	
&	$%A = 
\left(\begin{array}{cccccc}
	\ 1	&		&	&&&	\\
		&	-1	&	 &&&	\\
		&		&	0 	&&&\\
		& 	& & 0 && \\
		& 	& &&-1& \\
		& 	& & && 0\\
\end{array}\right)$	&	
$%B = 
\left(\begin{array}{cccccc}
	\ -1	&		&	&&&	\\
		&		1&	&&&	\\
		&		&	0 	&&&1\\
		& 	& &  0&& \\
		& 	& & &1& \\
		& 	& & && 0\\
\end{array}\right)$	&	
$\si^{11} \in F$		\\
\hline%\hline
%%%K_1,2(-1) Case 1
%\label{K12(-1)a}$K_{1,2}(-1)$	&	
(4)	&
$%A = 
\left(\begin{array}{cccccc}
	0	&		&		&        	&		&	\\
		&	1	&		& 		&    		&	\\
		&		&	-1 	&   		&		&	\\
		& 		&		&     1	&    	 	&	 \\
		& 		&          	&		& 0		&	 \\
		& 	   	& 		&	 	&		& 0	\\
\end{array}\right)$	&	
$%B = 
\left(\begin{array}{cccccc}
	0	&		&		&        	&		&1	\\
		&	-1	&		& 		&    		&	\\
		&		&	1 	&   		&		&	\\
		& 		&		&     -1 	&    	 	&	 \\
		& 		&          	&		&  0		&	 \\
		& 	   	& 		&	 	&		& 0	\\
\end{array}\right)$	&	$\si^{11} \in F$		\\
\hline
%%%K_1,2(-1) Case 2
%\label{K12(-1)b}		&	
(5)	&
$%A = 
\left(\begin{array}{cccccc}
	0	&		&		&        	&		&	\\
		&	1	&		& 		&    		&	\\
		&		&	-1 	&   		&		&	\\
		& 		&		&     1 	&    	 	&	 \\
		& 		&          	&		&0		&	 \\
		& 	   	& 		&	 	&		& 0	\\
\end{array}\right)$	&	$B = -A$	&	$\si^{11} \neq 0 \in F$		\\
\hline%\hline
%%%K_1,5 Case 1
%\label{K15a}$K_{1,5}$	&	
(6)	&	
$%A = 
\left(\begin{array}{cccccc}
	0	&		&		&        	&	1	&	\\
		&	1	&		& 		&    		&	\\
		&		&	-1 	&   		&		&	\\
		& 		&		&     1 	&    	 	&	 \\
		& 		&          	&		&     0	&	 \\
		& 	   	& 		&	 	&		& 0	\\
\end{array}\right)$	&	$B = -A$	&	$\si^{11} \neq 0 \in F$		\\
\hline
%%%K_1,5 Case 2
%\label{K15b}	&	
(7)	&
$%A = 
\left(\begin{array}{cccccc}
	0	&		&		&        	&	1	&	\\
		&	1	&		& 		&    		&	\\
		&		&	-1 	&   		&		&	\\
		& 		&		&     1 	&    	 	&	 \\
		& 		&          	&		&0		&	 \\
		& 	   	& 		&	 	&		& 0	\\
\end{array}\right)$	&
$%B = 
\left(\begin{array}{cccccc}
	0	&		&		&        	&	-1	& 1	\\
		&	-1	&		& 		&    		&	\\
		&		&	1 	&   		&		&	\\
		& 		&		&     -1 	&    	 	&	 \\
		& 		&          	&		& 0		&	 \\
		& 	   	& 		&	 	&		& 0	\\
\end{array}\right)$	&	$\si^{11} \in F$		\\
\hline%\hline
%%%K_1,6(0) 
%\label{K16(0)}$K_{1,6}(0)$		&	
(8) &	
$%A = 
\left(\begin{array}{cccccc}
	1	&		&		&        	&		&	\\
		&	0	&		& 		&    		& 1	\\
		&		&	-1 	&   		&		&	\\
		& 		&		&     1 	&    	 	&	 \\
		& 		&          	&		&    -1	&	 \\
		& 	   	& 		&	 	&		&   0	\\
\end{array}\right)$	&	
$%B = 
\left(\begin{array}{cccccc}
	-1	&		&		&        	&		&  	\\
		&	0	&		& 		&    		& b	\\
		&		&	1 	&   		&		&	\\
		& 		&		&     -1 	&    	 	&	 \\
		& 		&          	&		&  1		&	 \\
		& 	   	& 		&	 	&		& 0	\\
\end{array}\right)$	&	%$\si^{11} \in F$	
$\begin{array}{c}
\si^{11}, b \in F\\
\\
\si^{11}, b+1 \text{ not both zero}
\end{array}
$	\\
\hline%\hline
%%%K_1,8(-1) Case 1
%\label{K18(-1)a}$K_{1,8}(-1)$	&	
(9)	&
$%A = 
\left(\begin{array}{cccccc}
	1	&		&		&        	&		&	\\
		&	-1	&		& 		&    		& 	\\
		&		&	0	&   	1	&		&	\\
		& 		&		&     0 	&    	 	&	 \\
		& 		&          	&		&    -1	&	 \\
		& 	   	& 		&	 	&		&   0	\\
\end{array}\right)$	&	$B= -A$	&	$\si^{11} \neq 0 \in F$		\\
\hline
%%%K_1,8(-1) Case 2
%\label{K18(-1)b}&	
(10)	&
$%A = 
\left(\begin{array}{cccccc}
	1	&		&		&        	&		&	\\
		&	-1	&		& 		&    		& 	\\
		&		&	0	&   	1	&		&	\\
		& 		&		&     0 	&    	 	&	 \\
		& 		&          	&		&    -1	&	 \\
		& 	   	& 		&	 	&		&   0	\\
\end{array}\right)$	&
$%B = 
\left(\begin{array}{cccccc}
	-1	&		&		&        	&		&	\\
		&	1	&		& 		&    		& 	\\
		&		&	0	&   	-1	&		& 1	\\
		& 		&		&     0 	&    	 	&	 \\
		& 		&          	&		&    1		&	 \\
		& 	   	& 		&	 	&		&   0	\\
\end{array}\right)$	&	$\si^{11} \in F$		\\
\hline\hline
\end{tabular}

\label{L41}
\end{table}

%%%%%%%%%%%%%%%%%%%%
%%		Tabular Version of L(4,2)	%%
%%%%%%%%%%%%%%%%%%%%
\begin{table}[th]
\tiny
\caption{The Leibniz algebras $L(4,2)$ of non-Lie type}
%\centering
\begin{tabular}{lllc}  %%% INSERTED L FOR NUMBERING
\hline
%Name	&	No.		&	$A$	&	$B$	&	$\si$	\\
No. &  %%% INSERTED FOR NUMBERING
$A^1=-B^1$ & $A^2=-B^2$ & $\si$ \\
\hline\hline
%%%K_2,2 Case 1
%$K_{2,2}$ \label{K22a1}	&	(1)	&
%$A^1 = 
(11) &  %%% INSERTED FOR NUMBERING
$\left(\begin{array}{cccccc}
	1	&		&		&   &		&	  \\
		&	0	&		& 	&   & 	\\
		&		&	-1&   &		&	  \\
		& 	&		& 1 &   &	  \\
		& 	&  	&		& -1&	  \\
		&  	& 	&	 	&		& 0	\\
\end{array}\right)$	
&
$\left(\begin{array}{cccccc}
	0	&		&		&   &		&	  \\
		&	1	&		& 	&   & 	\\
		&		&	-1&   &		& 	\\
		& 	&		& 1	&   &	  \\
		& 	&   &		& 0	&	  \\
		& 	& 	&	 	&		& 0	\\
\end{array}\right)$
&
$\begin{array}{c}
\si^{11}, \si^{22}, \si^{12}, \si^{21} \in F\\
\\
\si^{11}, \si^{22}, \si^{12}+\si^{21}\text{ not all zero}
\end{array}
$ \\
\hline\hline
\end{tabular}
\label{L42}
%\ 	\\
%\begin{tabular}{c}
%	\\
%There are no 9 dimensional non-Lie solvable Leibniz algebras with $\nr(L)=T(4)$.\\
%	\\
%\hline\hline
%\end{tabular}
\end{table}

\end{document}